\documentclass[12pt]{amsart}

\usepackage{amsmath,graphicx}
\usepackage{adjustbox}
\usepackage{amssymb}
\usepackage{tikz, tikz-cd}
\usepackage{fullpage}
\usepackage[all]{xy}
\usepackage[margin=1in]{geometry}
\usepackage{amsthm}
\usepackage{amsfonts}
\usepackage{bbm}
\usepackage{eufrak}
\usepackage{comment}
\usepackage{amsmath}
\usepackage{amsthm}
\usepackage{bbm}
\usepackage{eufrak}
\usepackage{array} 
\usepackage{color}
\usepackage[utf8]{inputenc}
\usepackage[english]{babel}
\usepackage{hyperref}
\usepackage{braids}

\newcolumntype{L}{>{$}l<{$}}

\DeclareMathSymbol{\shortminus}{\mathbin}{AMSa}{"39}

\newtheorem{theorem}{Theorem}[section]
\newtheorem{lemma}[theorem]{Lemma}

\newtheorem{corollary}[theorem]{Corollary}

\newtheorem{proposition}[theorem]{Proposition}
\theoremstyle{definition}

\newtheorem{remark} [theorem] {Remark}

\theoremstyle{definition}

\newcommand{\F}{{\mathbb{F}}}

\newcommand{\Q}{{\mathbb{Q}}}

\newcommand{\Z}{{\mathbb{Z}}}

\newcommand{\lk}{\ell\text{k}}

\newcommand{\spin}{\text{spin}}

\DeclareMathOperator{\HFL}{HFL}
\DeclareMathOperator{\HFK}{HFK}

\DeclareMathOperator{\HF}{HF}

\DeclareMathOperator{\SFH}{SFH}

\DeclareMathOperator{\rank}{rank}

\DeclareMathOperator{\PD}{PD}

\title{Cable Links, Annuli and Sutured Floer homology}

\author[Binns]{Fraser Binns}
\author[Dey]{Subhankar Dey}
\address[]{Department of Mathematics, Boston College}
\email{binnsf@bc.edu}
\address[]{Department of Mathematics, University of Alabama}
\email{sdey4@ua.edu}
\date{}

\begin{document}

\maketitle

\begin{abstract}
    We apply sutured Floer homology techniques to study the knot and link Floer homologies of various links with annuli embedded in their exteriors. Our main results include, for large $m$, characterizations of links with the same link Floer homology as $(n,nm)$-cables of $L$-space knots or the same knot Floer homology as $(2,2m)$-cables of $L$-space knots. Note that Knot Floer homology carries less grading data than Link Floer homology, so the latter characterizations are stronger than former. These characterizations yield some new link detection results.
\end{abstract}

\section{Introduction}

Knot Floer homology is a powerful link invariant due independently to Ozsv\'ath-Szab\'o~\cite{Holomorphicdisksandknotinvariants} and J.Rasmussen~\cite{Rasmussen}. Much attention has been devoted to the behavior of this invariant under the operation of cabling, under which a torus link is tied into a prescribed knot, see~\cite{hanselman2019cabling,hedden2005knot,gorsky2017cable,dey2019cable}. In this note we consider a question in the opposite direction, namely: if we know a link has the homology type of a cable, can we conclude anything about the link? We will be especially interested in the case of $(n,mn)$-cables of \emph{L-space knots}. For a definition of $L$-space knots see Section~\ref{SFHbackground}. Here we use the convention that the meridional wrapping number is $mn$, while the longitudinal wrapping number is $n$. If $n=0$ then any $(n,mn)$-cable is the empty link, so we take $n\neq 0$ throughout this paper. Unless stated otherwise we also take $n>0$, primarily to state our results more concisely. See Remark~\ref{rmk:negativen} for some discussion.

Our main result is the following:

\newtheorem*{HFLmmn}{Theorem~\ref{HFLmmn}}
\begin{HFLmmn}
Suppose $K$ is a non-trivial $L$-space knot, $m> 2g(K)-1$ and $L$ is a link with $\widehat{\HFL}(L)\cong\widehat{\HFL}(K_{n,mn})$. Then $L$ is isotopic to $K'_{n,nm}$ for some $L$-space knot $K'$ such that $\widehat{\HFK}(K')\cong\widehat{\HFK}(K)$.
\end{HFLmmn}

The main ingredients in the proof of this theorem are the fact that link Floer homology detects the Thurston norm~\cite[Theorem 1.1]{ozsvath2008linkFloerThurstonnorm} and Juh\'asz's surface decomposition formula for sutured Floer homology~\cite[Theorem 1.3]{juhasz2008floer}. This yields some detection results as immediate corollaries:

\newtheorem*{HFLdetects}{Corollary~\ref{HFLdetects}}
\begin{HFLdetects}
Link Floer homology detects:\begin{enumerate}
    \item $T(2,3)_{n,mn}$ for $m> 1$.
    \item $T(2,5)_{n,mn}$ for $m> 3$.
\end{enumerate}
\end{HFLdetects}

 The $n=2$ case of item 1 in Corollary~\ref{HFLdetects} was shown in~\cite[Theorem 5.1]{binns2022rank}. Corollary~\ref{HFLdetects} follows from the fact that $T(2,3)$ and $T(2,5)$ are $L$-space knots and knot Floer homology detects $T(2,3)$~\cite[Corollary 1.5]{ghiggini2008knot}, and $T(2,5)$~\cite[Theorem A]{farber2022fixed}, respectively. Indeed, since $K_{1,m}$ is isotopic to $K$, the $n=1$ case of this corollary is exactly the result that knot Floer homology detects the $T(2,3)$ and $T(2,5)$.

An almost identical approach, combined with the fact that knot Floer homology detects the unknot~\cite{ozsvath2004holomorphic}, yields the following:

\newtheorem*{Tnnmdetection}{Theorem~\ref{Tnnmdetection}}
\begin{Tnnmdetection}
Link Floer homology detects $T(n,nm)$ for all $0\neq n\in \Z,m\in\Z$.
\end{Tnnmdetection}
This strengthens $T(n,n)$ detection, a result given in~\cite[Theorem 6.1]{binns2020knot}. Note again that the $n=1$ case of this theorem is exactly the result that knot Floer homology detects the unknot.

Using a variation of the proof of Theorem~\ref{HFLmmn} we obtain the following stronger statement for $(2,2n)$-cables, oriented as the boundary of annuli:

\newtheorem*{HFK22n}{Theorem~\ref{HFK22n}}
\begin{HFK22n}
Suppose $K$ is a non-trivial $L$-space knot, $n>2g(K)-1$ and $L$ is a link with $\widehat{\HFK}(L)\cong\widehat{\HFK}(K_{2,2n})$. Then there is an $L$-space knot $K'$ with $\widehat{\HFK}(K')\cong\widehat{\HFK}(K)$ such that $L$ is isotopic to $K'_{2,2n}$.
\end{HFK22n}

This result is stronger than Theorem~\ref{HFLmmn} because Link Floer homology carries less gradings information -- see Section~\ref{SFHbackground} for details. In particular, rather than detecting the Thurston norm of a link, knot Floer homology detects only the genus of the link~\cite[Theorem 1.1]{ni2006note}, \cite[Theorem 1.2]{ozsvath2004holomorphic}. Indeed, since the Knot Floer homology of a link with a given number of components can be recovered from Link Floer homology, Theorem~\ref{HFK22n} implies Theorem~\ref{HFLmmn} in the cases it applies. As with Theorem~\ref{HFLdetects}, Theorem~\ref{HFLmmn} has some detection results as immediate corollaries:

\newtheorem*{HFKdetects}{Corollary~\ref{HFKdetects}}
\begin{HFKdetects}
Knot Floer homology detects:\begin{enumerate}
    \item $T(2,3)_{2,2n}$ oriented as the boundary of an annulus for $n> 1$
    \item $T(2,5)_{2,2n}$ oriented as the boundary of an annulus for $n> 3$.
\end{enumerate}
\end{HFKdetects}
 These follow from Theorem~\ref{HFK22n} in the same way that Corollary~\ref{HFLdetects} follows from Proposition~\ref{HFLLspace}. This corollary can be viewed as an extension of $T(2,2n)=T(2,1)_{2,2n}$ detection, as shown in~\cite[Theorem 3.2]{binns2020knot}. We note that the proofs of Theorem~\ref{HFK22n}, Theorem~\ref{HFLmmn}, Theorem~\ref{Tnnmdetection} Corollary~\ref{HFKdetects} and Corollary~\ref{HFLdetects} only use the Alexander gradings and the reduction of the absolute Maslov grading mod two.

Similar techniques can be applied to understand the link Floer homologies of links with high Euler characteristic surfaces embedded in their exteriors. Let $U$ be the unknot, and $\sqcup$ denote the split sum operation.

\newtheorem*{HFLunknot}{Proposition~\ref{HFLunknot}}
\begin{HFLunknot}
 Suppose $L'$ is a link such that $\widehat{\HFL}(L')\cong\widehat{\HFL}(L\sqcup U)$. Then $L'$ is isotopic to $L''\sqcup U$ where $L''$ is a link such that $\widehat{\HFL}(L'')\cong\widehat{\HFL}(L)$.
\end{HFLunknot}
The authors suppose that this result was already known, but were unable to find a reference in the literature.

Let $H_\pm$ denote the positive or negative Hopf link.
\newtheorem*{HFLhopf}{Proposition~\ref{HFLhopf}}
\begin{HFLhopf}

Suppose $L'$ is a link such that $\widehat{\HFL}(L')\cong\widehat{\HFL}(L\#H_\pm)$. Then $L'$ is isotopic to $L''\# H_\pm$ where $L''$ is a link such that $\widehat{\HFL}(L'')\cong\widehat{\HFL}(L)$.
\end{HFLhopf}

This should be compared with a very similar result~\cite[Proposition 9.2]{binns2020knot}. Note however that there is no hypothesis on Alexander gradings in the above statement. The main novelty of Proposition~\ref{HFLunknot} and Proposition~\ref{HFLhopf} is the fact that their proofs uses Juh\'asz's sutured decomposition formula~\cite[Theorem 1.3]{juhasz2008floer}. We note that Propositions~\ref{HFLunknot} and~\ref{HFLhopf} can be used to produce infinite families of links which knot Floer Homology and Khovanov homology cannot distinguish but which link Floer homology can detect, as in~\cite[Theorem 9.4]{binns2020knot}. For example, these two propositions imply that link Floer homology detects all forests of unlinks -- that is links which arise from unlinks via iterated connect sums -- while the knot Floer homology and Khovanov homology are determined by the number of components and number of split components alone. 

Our final result is of a distinct flavour and uses different methods than those preceding. View $T(2,n)$ as the closure of a $2$-braid. Denote by $L_n'$ the link formed by taking $T(2,n)$ union its braid axis as shown in Figure~\ref{fig:Ln}, but with arbitrary orientation.
\newtheorem*{2braid}{Theorem~\ref{2braid}}
\begin{2braid}
Link Floer homology detects $L_n'$ for all $n\in\Z$.
\end{2braid}

This result can be thought of as the Floer homology analogue of the corresponding result that Annular Khovanov homology detects all $2$-braids -- which follow from computations of Grigsby-Licata-Wehrli~\cite[Proposition 15]{grigsby_annular_2018} and the fact that annular Khovanov homology detects braids and the braid index by results of Grigsby-Ni~\cite[Corollary 1.2]{grigsby_sutured_2014}. The $n=0$ case is a special case of a result of Baldwin-Grigsby~\cite[Theorem 3.1]{BG}, while the $n=1,2$ and $3$ cases are proven in~\cite{binns2020knot}.

The outline of the paper is as follows: in Section~\ref{SFHbackground} we briefly review sutured Floer homology focusing on the properties that will be of use to us in subsequent sections. In Section~\ref{HFLmmnsection}, \ref{HFK22nsection} and Section~\ref{HFLhopfunknotsection}, we prove Theorem~\ref{HFLmmn}, Theorem~\ref{HFK22n} and related results. We conclude by proving Theorem~\ref{2braid} in Section~\ref{2braidsection}.

	\subsection*{Acknowledgements}
The first author would like to thank his advisor John Baldwin for his ongoing support, as well as Gage Martin. He would also like to thank the audience from his talk at the GSTGC conference at Georgia Tech for prompting him to think harder about how Theorem~\ref{HFK22n} could be generalised, resulting in Theorem~\ref{HFLmmn}. We would also like to thank Tye Lidman, Jen Hom, Matt Hedden and Zhenkun Li for their feedback and questions on an earlier draft. We would also like to thank the referees for their detailed comments on the first version of this paper. The second author would like to acknowledge partial support of NSF grants DMS 2144363, DMS 2105525, and AMS-Simons travel grant.

\section{Preliminaries on Sutured, Link and Knot Floer Homology}\label{SFHbackground}

In this section we review aspects of sutured Floer homology which will be relevant in later sections. We refer the reader to~\cite{juhasz2006holomorphic} and~\cite{juhasz2008floer} for details.

A \emph{balanced sutured manifold} is a pair $(Y,\gamma)$ where  $Y$ is a $3$-manifold such that each component of $Y$ has non-empty boundary and each component of $\partial Y$ contains a component of the \emph{suture}, $\gamma$. Here a suture is a (perhaps disconnected) embedded 1-manifold in $\partial Y$ which separates $\partial Y$ into two pieces of equal Euler characteristic. $Y$ is also required to satisfy some orientability conditions, see~\cite[Definition 2.1]{juhasz2006holomorphic} for details. This definition is due to Juh\'asz~\cite[Definition 2.2]{juhasz2006holomorphic}. A more general version of sutured manifolds were first studied by Gabai in the context of taut foliations~\cite{gabai1983foliations}.

Sutured Floer homology is an invariant of sutured manifolds defined by Juh\'asz~\cite{juhasz2006holomorphic}. Sutured Floer homology decomposes as a direct sum along relative $\spin^c$ structures. It behaves nicely under \textit{sutured manifold decompositions}, see~\cite{juhasz2008floer} for details. While sutured Floer homology and the other Floer theoretic invariants we study in this paper can be defined with coefficients in $\Z$, in this paper we take all coefficients to be in the field of two elements, $\F$, for simplicity. We note also that the pairing theorem for immersed curves \cite[Theorem 1.2]{hanselman2023bordered}, which we shall use in a couple of places, has only been proven with coefficients in $\F$.

We will be interested in a number of special cases of sutured Floer homology. The first of these is \emph{Link Floer homology}, which was originally due to Ozsv\'ath-Szab\'o~\cite{HolomorphicdiskslinkinvariantsandthemultivariableAlexanderpolynomial}. Link Floer homology can be thought of as the sutured Floer homology of a link exterior with pairs of meridional sutures. We call the $\spin^c$ grading on link Floer homology the \emph{Alexander grading}. If $L$ is an $n$-component link the Alexander grading can be thought of as an affine lattice over $\Z^n$ or indeed over $H_1(X(L))$. Link Floer homology detects the Thurston norm of a link exterior, see~\cite[Theorem 1.1]{ozsvath2008linkFloerThurstonnorm} for a precise statement.

The second special case of sutured Floer homology that we will be interested in is \emph{knot Floer homology}, which is due independently to Ozsv\'ath-Szab\'o~\cite{ozsvath2004holomorphic} and J. Rasmussen~\cite{Rasmussen}. Knot Floer homology can be thought of as the sutured Floer homology of a knot exterior with a pair of parallel oppositely oriented meridional sutures. We again call the $\spin^c$ grading on knot Floer homology the \emph{Alexander grading}, and can think of it as taking values in $\Z$. We note that knot Floer homology can be extended to an invariant of links by the process of \emph{knotification}. This process takes an $n$ component link in $Y$ and yields a knot in $Y\#^{n-1}(S^1\times S^2)$ ~\cite[Subsection 2.1]{Holomorphicdisksandknotinvariants}. An alternate description of Knot Floer homology is discussed in Remark~\ref{rem:HFKalt}. Important properties of knot Floer homology include that it detects the maximal Euler characteristic of a Seifert surface for a link~\cite[Theorem 1.1]{ni2006note},\cite[Theorem 1.2]{ozsvath2004holomorphic} and that it categorifies the Alexander polynomial~\cite[Equation 1]{Holomorphicdisksandknotinvariants}. It follows that link Floer homology detects the linking number of two component links by a result of Hoste~\cite[Theorem 1]{hoste1985firstcoefficientoftheconwaypolynomial}. We will indeed use a stronger version of Hoste's result in Section~\ref{HFLmmnsection} and Section~\ref{2braidsection}.

Let $K_n$ denote the core of $n$-surgery on a knot $K$. An \emph{$L$-space knot} is a knot $K$ for which $\rank(\widehat{\HFK}(K_n))=n$ for some $n\geq 0$. Note that this is a non-standard definition of an $L$-space knot, but follows quickly from the immersed curve interpretation for the surgery formula in bordered Floer homology. More specifically, the immersed curve interpretation of bordered Floer homology allows one to compute $\rank(\widehat{\HFK}(K_n))$ by counting the intersections of two multi-curve in a torus: one of slope $n$ -- under an appropriate parametrization -- and a multi-curve determined by $K$~\cite[Remark 52]{hanselman2018heegaard}. In the case that $K$ is an $L$-space knot the multi-curve associated to $K$ has a particularly simple form, see~\cite[Section 7.5]{hanselman2023bordered} for details. Alternatively this follows from work of Hedden~\cite[Theorem 4.1]{hedden2007knot}, \cite[Theorem 1.4]{hedden2011floer}, or Eftekhary~\cite[Theorem 1.2]{eftekhary2011knots}. The unknot, $U$, is an example of an $L$-space knot. Note that $\rank(\widehat{\HFK}(U_n))=|n|$ for all $n$. If $K$ is a non-trivial $L$-space knot then $\rank(\widehat{\HFK}(K_n))=n$ if and only if $n> 2g(K)-1$, again see~\cite[Remark 52]{hanselman2018heegaard} for details. $L$-space knots have many strong properties, including that their knot Floer homologies are determined by their Alexander polynomials.

\section{Link Floer homology and $(n,nm)$-cables}\label{HFLmmnsection}

In this section we prove the main theorem of this paper. Throughout this section we give $K_{n,nm}$ the orientation induced by an orientation on $K$.

\begin{theorem}\label{HFLmmn}
Suppose $K$ is a non-trivial $L$-space knot, $m> 2g(K)-1$ and $L$ is a link with $\widehat{\HFL}(L)\cong\widehat{\HFL}(K_{n,mn})$. Then $L$ is isotopic to $K'_{n,nm}$ for some $L$-space knot $K'$ such that $\widehat{\HFK}(K')\cong\widehat{\HFK}(K)$.
\end{theorem}

The above result is also true in the case that $K$ is the unknot:

\begin{theorem}\label{Tnnmdetection}
Link Floer homology detects $T(n,nm)$ for all $n,m$.
\end{theorem}

We have separated these two results purely for expository purposes. Taking $K$ to be the trefoil or cinquefoil and combining Theorem~\ref{HFLmmn} with the corresponding detection results for knot Floer homology due to Ghiggini~\cite[Corollary 1.5]{ghiggini2008knot} and Farber-Reinoso-Wang~\cite[Theorem A]{farber2022fixed} implies the following:

\begin{corollary}\label{HFLdetects}
Link Floer homology detects:\begin{enumerate}
    \item $T(2,3)_{n,mn}$ for $m> 1$.
    \item $T(2,5)_{n,mn}$ for $m> 3$.
\end{enumerate}
\end{corollary}

The outline of the proof of Theorem~\ref{HFLmmn} is as follows. Using the fact that link Floer homology detects the Thurston polytope~\cite[Theorem 1.1]{ozsvath2008linkFloerThurstonnorm}, we obtain a family of annuli embedded in $X(L)$ along which we perform a sutured decomposition. Understanding this manifold allows us to use Juh\'asz's result~\cite[Theorem 1.3]{juhasz2008floer} to deduce information about $X(L)$.

We begin by proving the link Floer homology of $(n,nm)$-cables of $L$-space knots have certain properties. Set $\widehat{\HFL}(L,I_k):=\underset{x\in I_k}{\bigoplus}\widehat{\HFL}(L,x)$ where $$I_k=\{(A_1,A_2,\dots ,A_n):k(n-1)=(n-1)A_n-\underset{1\leq i\leq n-1}{\sum}A_i)\}.$$

\begin{proposition}\label{HFLdecompositionrank}
Let $K$ be a non-trivial $L$-space knot with $m> 2g(K)-1$. Then: $${\rank(\widehat{\HFL}(K_{n,nm},I_{1}))=2^{n-2}m}.$$
\end{proposition}

We note that this proposition can be deduced from a complete computation of the link Floer homology due to Gorsky-Hom~\cite[Theorem 3]{gorsky2017cable}. However, we give a proof using Juh\'asz's sutured decomposition formula~\cite[Theorem 1.3]{juhasz2008floer}, which fits more naturally with the perspective we take in this paper.

We also fix some notation. Let $L$ be a link with $n$ components $\{L_i\}_{1\leq i\leq n}$. $H_1(\partial X(L))$ is generated by the homology classes of curves $\{\mu_i,\lambda_i\}_{1\leq i\leq n}$ where $\mu_i$ and $\lambda_i$ are respectively meridians and longitudes of $L_i$.

\begin{proof}[Proof of Proposition~\ref{HFLdecompositionrank}]
Let $K$ be an $L$-space knot. In $X(K_{n,nm})$ there is a family of annuli, $\{A_i\}_{1\leq i\leq n-1}$, where $A_i$ has boundary representing $-m[\mu_i]-[\lambda_i]+m[\mu_n]+[\lambda_n]$ in $H_1(\partial X(K_{n,nm}))$, for $1\leq i\leq n-1$. Each annulus $A_i$ can be realized by taking the trace of an isotopy from the $i\neq n$th component and the $n$th component of $K_{n,nm}$. Decomposing $X(K_{n,nm})$ along $\Sigma=\underset{1\leq i\leq n-1}{\bigcup} A_i$ yields the manifold $X(K)$ with $n-1$ parallel pairs of sutures of slope $m[\mu]+[\lambda]$, where $\mu$ and $\lambda$ are the meridian and longitude of $K$ respectively. Call this sutured manifold $(Y,\gamma)$. Since $(Y,\gamma)$ is obtained by adding $n-2$ pairs of parallel sutures to the exterior of $K$ in $S^3$ with a pair of parallel sutures of slope $n$, it follows from~\cite[Proposition 9.2]{juhasz2010sutured} that $\SFH(Y,\gamma)$ is given by $\widehat{\HFK}(K_m)\otimes V^{\otimes(n-2)}$, where $V$ is a rank $2$ vector space. Thus $\SFH(Y,\gamma)$ is of rank $2^{n-2} \cdot m$. Juh\'asz's theorem~\cite[Theorem 1.3]{juhasz2008floer} then implies that $\rank(\widehat{\HFL}(K_{n,nm},I_{1}))=2^{n-2}m$. Here the fact that $k=1$ follows from the fact that $\chi(\Sigma)=0$, $r(A_i,t_0)=0$ for all $i\in\{1,2,\dots, n\}$ and $I(\Sigma)=0$, following Juh\'asz's notation from~\cite[Lemma 3.9]{juhasz2008floer}.
\end{proof}

We can prove a stronger version of Proposition~\ref{HFLdecompositionrank} for the unknot, $U$.
\begin{lemma}\label{unknotcomp}
${\rank(\widehat{\HFL}(T(n,nm),I_{1}))=2^{n-2}|m|}$ for all $m\neq 0$.
\end{lemma}

\begin{proof}
If $m\neq 0$ then the proof follows exactly as in the proof of Lemma~\ref{HFLdecompositionrank}. $T(n,0)$ is the $n$-component unlink, for which it is readily computed that $\rank(\widehat{\HFL}(T(n,0),I_{1}))=0$.
\end{proof}

The bulk of the work in proving Theorem~\ref{HFLmmn} is contained in the proof of the following proposition.

\begin{proposition}~\label{HFLLspace}
Let $K$ be a non-trivial $L$-space knot and $m> 2g(K)-1$. Suppose that $L$ is a link such that $\widehat{\HFL}(L)\cong\widehat{\HFL}(K_{n,mn})$. Then $L$ is the $(n,mn)$-cable of an $L$-space knot $K'$.
\end{proposition}

\begin{proof}
Suppose $\widehat{\HFL}(L)$ is as in the statement of the theorem. We first note that $L$ does not contain any split unknotted components, since for each $i$ there exists a generator of link Floer homology with $i$th Alexander grading nonzero. Since link Floer homology detects the Thurston norm, $L$ and $K_{n,mn}$ have equivalent Thurston norm. We thus have that there is an Euler characteristic $0$ surface, denoted by $\Sigma$, with: \begin{equation}
[\Sigma]=(n-1)\PD[\mu_n]-\underset{1\leq i\leq n-1}{\sum}\PD[\mu_i]\in H_2(X(L)),\partial (X(L)))\end{equation}\label{eq:homboundary} 

Here we view $[\mu_i]$ as generators of $H^1(X(L))$ corresponding to meridians of each component of $L$. Since $L$ does not have split unknot components, $\Sigma$ must be a collection of embedded annuli. Let $L_i$ denote the $i$th component of $L$.

We now seek to simplify $\Sigma$ while preserving $[\Sigma]\in H_2(X(L),\partial(X(L)))$ so that on each boundary component of $X(L)$, the components of $\partial \Sigma$ are coherently oriented. If $A_1, A_2$ are two connected components of $\Sigma$ with boundary components on $\partial(\nu(L_i))$ for some $i$ then $[\partial A_1|_{\partial\nu(L_i)}]=\pm[\partial A_2|_{\partial\nu(L_i)}]\in H_1(\partial\nu(L_i))$. Suppose there are annuli $A_1, A_2$ for which $[\partial A_1|_{\partial\nu(L_i)}]=-[\partial A_2|_{\partial\nu(L_i)}]\in H_1(\partial\nu(L_i))$. We may assume these are adjacent in the sense that $\partial\nu(L_i)-(\Sigma-A_1-A_2)$ contains a path from $A_1$ to $A_2$. We can form a surface $\Sigma'$ with one fewer component than $\Sigma$ representing $[\Sigma']=[\Sigma]\in H_2(X(L)),\partial (X(L)))$ by merging the two annuli near $\partial\nu(L_i)$. Iterating we see that without loss of generality we may take all components of $\partial \Sigma$ to carry the same orientation in each $\partial\nu(L_i)$. We relabel this surface as $\Sigma$.

Equation~\ref{eq:homboundary} implies that $\Sigma$ must have at least $n-1$ boundary components on $\partial(X(L_n))$ and at least one boundary component on each of $\partial(X( L_i))$ for $1\leq i\leq n-1$. Since the components of $\partial \Sigma$ are all oriented in the same direction on $\partial\nu(L_i)$ for each $i$, it follows that $\Sigma$ must indeed have exactly one boundary component on each of $\partial(X(L_i))$ for each $1\leq i\leq n-1$ and exactly $n-1$ boundary components on $\partial(X(L_n))$. It follows in turn that $\Sigma$ consists of $(n-1)$ annuli, $A_i, 1\leq i\leq n-1,$ where $A_i$ has a boundary component on $\partial(\nu(L_n))$ and another on $\partial(\nu(L_i))$ with $i\neq n$. It follows that $L$ is a $(n,pn)$-cable of $L_n$ for some $p$. In particular $\lk(L_i,L_j)=p$ for all $i\neq j$.

Now, by a result of Hoste~\cite[Theorem 1]{hoste1985firstcoefficientoftheconwaypolynomial}, if $L'$ is a link with components $L'_i$ then the Conway polynomial determines the co-factors of the matrix $M(L')$ with entries given by $l_{i,j}=\lk(L_i',L_j')$ for $i\neq j$, and $l_{i,i}=-\sum_{1\leq j\leq n, j\neq i}\lk(L_i',L_j')$. Consider the $(n,qn)$-cable of some knot $K'$. Observe that any co-factor of $M(K'_{n, qn})$ is given by $q^{n-1}\det(A)$ where $A$ is defined by $a_{i,j}=1$ for $i\neq j$, $a_{i,i}=1-n$. Note that $\det(A)\neq 0$. Since Link Floer homology determines the Conway polynomial, $K_{n,mn}$ and the $(n,pn)$-cable of $L_n$ have the same Conway polynomial. It follows in turn that $p=m$.

Let $K'=L_n$. Decomposing the exterior of the $(n,mn)$-cable of $K'$ along $\Sigma$, we find a sutured manifold with underlying smooth manifold given by $X(K')$. The sutures are given by $2(n-1)$ parallel copies of curves of slope $m$. By~\cite[Proposition 9.2]{juhasz2010sutured}, this has sutured Floer homology given by $\widehat{\HFL}(L,I_1)\cong\widehat{\HFK}(K')\otimes V^{\otimes(n-2)}$. It follows that $\widehat{\HFK}(K')\otimes(\F^{2})^{\otimes(n-2)}$ is of rank $2^{n-2}m$ whence $K'$ is an $L$-space knot.
\end{proof}

Proposition~\ref{HFLLspace} is also true in the case that $K$ is trivial:

\begin{lemma}\label{unknotlspace}
Suppose $\widehat{\HFL}(L)\cong\widehat{\HFL}(T(n,nm))$ for some $m\in \Z$. Then $L$ is the $(n,nm)$-cable of an $L$-space knot.
\end{lemma}

\begin{proof}
If $m\neq0$ the result follows from Lemma~\ref{unknotcomp} just as Proposition~\ref{HFLLspace} follows from Proposition~\ref{HFLdecompositionrank}. $T(n,0)$ is the $n$ component unlink which Link Floer homology is known to detect.
\end{proof}

To conclude the proof of Theorem~\ref{HFLmmn} we will need the following lemma:
\begin{lemma}\label{alexanderlemma}
If $K$ and $K'$ are $L$-space knots with $\Delta_{K_{n,mn}}(t)=\Delta_{K'_{n,mn}}(t)$ for any $m\in \Z$ then $\widehat{\HFK}(K')\cong\widehat{\HFK}(K)$. 
\end{lemma}

\begin{proof}
Suppose $K$ and $K'$ are as in the statement of the lemma. Then $$\Delta_K(t)=\Delta_{K'}(t)=\dfrac{\Delta_{K_{n,mn}}(t)}{\Delta_{T(n,mn)}(t^n)}.$$ See~\cite[Theorem 6.15]{lickorish2012introduction}, where the statement is given in the case that $K_{n,mn}$ is a knot. The proof still holds in the link case. Since $K'$ and $K$ are $L$-space knots $\widehat{\HFK}(K')\cong\widehat{\HFK}(K)$.
\end{proof}

We can now conclude the proof of the main theorem.

\begin{proof}[Proof of Theorem~\ref{HFLmmn}]
Suppose $L,K$ are is in the statement of the theorem. Then $L$ is an $(n,mn)$-cable of an $L$-space knot $K'$ by Proposition~\ref{HFLLspace}, whence Lemma~\ref{alexanderlemma} implies that $\widehat{\HFK}(K')\cong\widehat{\HFK}(K)$, as required.
\end{proof}

$T(m,mn)$ detection also follows quickly:

\begin{proof}[Proof of Theorem~\ref{Tnnmdetection}]
Suppose $\widehat{\HFL}(L)\cong\widehat{\HFL}(T(m,mn))$. Then Lemma~\ref{unknotlspace} implies $L$ is the $(n,nm)$-cable of an $L$-space knot $K'$. Lemma~\ref{alexanderlemma} implies that $L$ has the same knot Floer homology as the unknot, whence $K'$ is the unknot by~\cite{ozsvath2004holomorphic}, as desired.
\end{proof}

We conclude by proving the remaining two families of detection results.

\begin{proof}[Proof of Corollary~\ref{HFLdetects}]

Suppose $L$ is a link with knot Floer homology of one of the two given types. Note that $T(2,3)$ and $T(2,5)$ are $L$-space knots. Thus Theorem~\ref{HFLmmn} implies $L$ is a $(n,nm)$-cable of a knot with the same knot Floer homology as $T(2,3)$, or $T(2,5)$ respectively. The result then follows from the fact that knot Floer homology detects both of these knots~\cite[Corollary 1.5]{ghiggini2008knot},~\cite[Theorem A]{farber2022fixed}.
\end{proof}

\section{Knot Floer homology and $(2,2n)$-cables of $L$-space knots}~\label{HFK22nsection} In this section we prove a stronger version of Theorem~\ref{HFLmmn} for $(2,2n)$-cables in the context of knot Floer homology. Throughout this section we orient $(2,2n)$-cables links in such a way that they bound an annulus.

\begin{theorem}\label{HFK22n}
Let $K$ be a non-trivial $L$-space knot with $n> 2g(K)-1$. Then  $\widehat{\HFK}(L)\cong\widehat{\HFK}(K_{2,2n})$ if and only if $L$ is the $(2,2n)$ cable of an $L$-space knot $K'$ with $\widehat{\HFK}(K')\cong\widehat{\HFK}(K)$.
\end{theorem}

Note that a stronger version of the above result is true for the unknot, namely that knot Floer homology detects $T(2,2n)$ for all $n$~\cite[Theorem 3.2]{binns2020knot}.

Again this yields some detection results as corollaries:

\begin{corollary}\label{HFKdetects}
Knot Floer homology detects:\begin{enumerate}
    \item $T(2,3)_{2,2n}$ oriented as the boundary of an annulus for $n> 1$
    \item $T(2,5)_{2,2n}$ oriented as the boundary of an annulus for $n> 3$
    \end{enumerate}
\end{corollary}

We begin by proving the knot Floer homology of $(2,2n)$-cables of $L$-space knots satisfy the following properties:

\begin{proposition}\label{prop:key}
Let $K$ be a non-trivial $L$-space knot, $n> 2g(K)-1$. Then $\widehat{\HFK}(K_{2,2n})$ satisfies the following: \begin{enumerate}
    \item $\max\{A:\widehat{\HFK}(K_{2,2n},A)\neq0\}=1$
    \item $\rank(\widehat{\HFK}(K_{2,2n},1))=n$.
\end{enumerate}

\end{proposition}

Again these conditions can be deduced from~\cite[Theorem 3]{gorsky2017cable}, but we provide a proof using Juh\'asz's surface decomposition theorem~\cite[Theorem 1.3]{juhasz2008floer} and a skein exact triangle for knot Floer homology.

\begin{remark}\label{rem:HFKalt}
    While the knot Floer homology of a multi-component link was originally defined via knotification, it can alternately be defined as the sutured Floer homology of the link exterior equipped with parallel pairs of oppositely oriented meridional sutures. In this context the Alexander grading can be defined by evaluating $\spin^c$ structures on the image of Seifert surface for $L$ in $H_2(X(L),\partial(X(L)))$. The correspondence between these two approaches is given in~\cite[Theorem 1.1]{HolomorphicdiskslinkinvariantsandthemultivariableAlexanderpolynomial}, see~\cite[Section 10]{HolomorphicdiskslinkinvariantsandthemultivariableAlexanderpolynomial} for further details.
\end{remark}

\begin{proof}
The maximal Euler characteristic surface bounding a $(2,2n)$-cable knot is an annulus if $K$ is non-trivial, whence condition 1 follows. Decomposing along that annulus gives the exterior of $K_n$, which has sutured Floer homology of rank $n$ since $K$ is an $L$-space knot and $n> 2g(K)-1$, proving condition 2.
\end{proof}

\begin{lemma}\label{lem:2comps}
    Suppose $L$ is a link with $\widehat{\HFK}(L)\cong\widehat{\HFK}(K_{2,2n})$ for some knot $K$. Then $L$ has two components.
\end{lemma}

\begin{proof}
    Suppose $L$ is as in the statement of the Lemma. Observe that $L$ and $K_{2,2n}$ have the same Conway polynomial. Also $K_{2,2n}$ has linking number $-n$ and two components. By~\cite[Theorem 1]{hoste1985firstcoefficientoftheconwaypolynomial}, the coefficient of $z$ in the Conway polynomial is $n\neq 0$. On the other hand, for an $m$ component link, the Conway polynomial has minimum possible degree $m-1$~\cite[Lemma 2.1]{hoste1984arf}. It follows that $L$ has at most $2$ components. Recall that there is a spectral sequence from $\widehat{\HFL}(L)$ to $\widehat{\HF}(\#^{m-1}(S^1\times S^2))$, where $m$ is the number of components of $L$~\cite[Lemma 3.6]{Holomorphicdisksandknotinvariants}. Moreover, $\rank(\widehat{\HF}(\#^{m-1}(S^1\times S^2)))$ is odd if and only if $m=1$~\cite[Lemma 9.1]{ozsvath2004Holomorphicdisksandtopologicalinvariantsforclosedthreemanifolds}. Thus since $\rank(\widehat{\HFK}(L))$ is even it follows that $L$ cannot be a knot, so that $L$ has exactly two components.
\end{proof}
We now prove Theorem~\ref{HFK22n} using the following Lemma, noting the key properties listed in Proposition~\ref{prop:key}.

\begin{lemma}\label{HFKLspace}
Suppose $K$ is a non-trivial $L$-space knot and $\widehat{\HFK}(L)\cong\widehat{\HFK}(K_{2,2n})$ with $n> 2g(K)-1$. Then $L$ is isotopic to $K'_{2,2n}$ where $K'$ is an $L$-space knot with $2g(K')-1<n$.
\end{lemma}

\begin{proof}
Suppose $K,n$ are as in the statement of the Theorem. By Lemma~\ref{lem:2comps}, we have that $L$ has two components.

Since the maximal Alexander grading is $1$ and $L$ has non-zero linking number, $L$ bounds an annulus. Decomposing along this annulus yields a sutured manifold with sutured Floer homology given by $\widehat{\HFK}(K'_n)$ for some knot $K'$. Note that $\rank(\widehat{\HFK}(K'_n))=n$, hence $K'$ is indeed an $L$-space knot with $2g(K')-1<n$. 
\end{proof}

\begin{proof}[Proof of Theorem~\ref{HFK22n}]
Suppose $K$, $L$ are as in the statement of the theorem. Lemma~\ref{HFKLspace} implies that $L$ is given by $K'_{2,2n}$ for some $L$-space knot $K'$. Applying Lemma~\ref{alexanderlemma} shows that $\widehat{\HFK}(K)\cong\widehat{\HFK}(K')$, as desired.
\end{proof}
 
\begin{proof}[Proof of Corollary~\ref{HFKdetects}]

Suppose $L$ is a link with knot Floer homology of one of the two given types. Note that $T(2,3)$ and $T(2,5)$ are $L$-space knots. Thus Theorem~\ref{HFK22n} implies $L$ is a $(2,2n)$-cable of a knot with the same knot Floer homology as $T(2,3)$ or $T(2,5)$ respectively. The result then follows from the fact that knot Floer homology detects each of these two knots~\cite[Corollary 1.5]{ghiggini2008knot},~\cite[Theorem A]{farber2022fixed}.
\end{proof}
We conclude this section with the following remark:
 \begin{remark}\label{rmk:negativen}
     Versions of Theorem~\ref{HFK22n} and Theorem~\ref{HFLmmn} as well as Corollary~\ref{HFKdetects} and Corollary~\ref{HFLdetects} can be obtained for $(n,mn)$ cables with $n<0$. The relevant proofs in this section and Section~\ref{HFLmmnsection} follow through in the $n<0$ case if we take $m<1-2g(K)$.
 \end{remark}

\section{Unlink and Hopf link Summands, and Link Floer Homology}\label{HFLhopfunknotsection}

In this section we study the link Floer homology of links obtained by adding meridional components or split unknots. We obtain results very similar to~\cite[Proposition 9.2]{binns2020knot} with techniques similar to those applied to prove Theorem~\ref{HFLmmn} and Theorem~\ref{HFK22n}.

Here we take $L$ to be a link, $H_\pm$ to be the positive and negative Hopf links and $U$ be the unknot.

\begin{proposition}\label{HFLunknot}
 Suppose $L'$ is a link such that $\widehat{\HFL}(L')\cong\widehat{\HFL}(L\sqcup U)$. Then $L'$ is isotopic to $L''\sqcup U$ where $L''$ is a link such that $\widehat{\HFL}(L'')\cong\widehat{\HFL}(L)$.
\end{proposition}

\begin{proposition}\label{HFLhopf}
Suppose $L'$ is a link such that $\widehat{\HFL}(L')\cong\widehat{\HFL}(L\#H_\pm)$. Then $L'$ is isotopic to $L''\# H_\pm$ where $L''$ is a link such that $\widehat{\HFL}(L'')\cong\widehat{\HFL}(L)$.
\end{proposition}

Note that the connect sum operation is not a well defined operation on links; it depends on the component to which you connect sum. We suppress this to ease notation.

Proposition~\ref{HFLhopf} is a version of~\cite[Proposition 9.2]{binns2020knot}. We note again that this yields infinite families of links that link Floer homology detects but neither Khovanov homology nor knot Floer homology detect -- forests of unlinks, for example.

Let $n(S,L)\geq 0$ denote the \emph{geometric intersection number} of a surface $S$ and an $n$ component link $L$. A result of Juh\'asz~\cite[Lemma 3.10]{juhasz2008floer} implies that the maximum not trivial $A_{1}$ grading of $\widehat{\HFL}(L)$ is given by

\begin{equation}\label{eq:maxA1grading}
\frac{1}{2}\max\{-2g(S)-n(S,L-L_1):\partial S=L_1\}\end{equation}
 More specifically, ~\cite[Lemma 3.10]{juhasz2008floer} states that the maximal non-trivial $A_1$ grading is given by $\frac{1}{2}(\chi(S')+I(S')-r(S',t))$, where these quantities are defined in  ~\cite[Lemma 34]{juhasz2008floer}, and $S'$ is the image of $S$ in the exterior of $L$. Equation~\ref{eq:maxA1grading} is the main technical tool we use to prove our propositions.

\begin{proof}[Proof of Proposition~\ref{HFLunknot}]

Suppose $L$ and $L'$ are as in the statement of the theorem. Order the components of $L\sqcup U$ so that the named unknot $U$ is the first component.

Since the maximal $A_1$ grading of $\widehat{\HFL}(L')$ is $0$, we see that $L'_1$ -- the first component of $L'$, $L'_1$ -- bounds a surface $S$ with $0=-2g(S)-n(S,L'-L_1')$. It follows that $S$ is a disk and $n(S,L'-L_1')=0$, whence $L'_1$ is a split unlinked component of $L'$.

It thus suffices to show $L''$ -- the link consisting of the components of $L'$ without $L'_1$ -- has the same knot Floer homology as $L$. But this is immediate since $$\widehat{\HFL}(L)\otimes V \cong \widehat{\HFL}(L\sqcup U)\cong V\otimes\widehat{\HFL}(L'')$$ where here $V$ is a rank $2$ vector space supported in multi-Alexander grading $0$. Thus, since we are working with vector spaces over $\Z/2$, $\widehat{\HFL}(L'')\cong\widehat{\HFL}(L)$, as required.
\end{proof}

The proof of Proposition~\ref{HFLhopf} requires a slightly more sophisticated analysis.

\begin{proof}[Proof of Proposition~\ref{HFLhopf}]

Suppose $L$ and $L'$ are as in the statement of the theorem. After relabeling we may take the 1st component of $L\# H_\pm$ to be the unknotted component of $H_\pm$ which remains a component in $L\# H_\pm$.

Since the maximal $A_1$ grading of $\widehat{\HFL}(L')$ is $1$, we see that $L'_1$ -- the first component of $L'$ -- bounds a surface $S$ with $-1=-2g(S)-n(S,L-L'_1)$. It follows that $S$ is a once punctured disk. It follows that $L'_1$ is the meridian of a component $L_i'$ of $L'$ where $i\neq 1$. That is, $L'$ is of the form $L''\#H_{\pm}$. Let $L_i''$ be the component of $L''$ corresponding to $L_i'$.

It thus suffices to show $L''$ has the same Link Floer homology as $L$. Note that $S\cap X(L)$ is an annulus, $A$. Decomposing along this annulus yields a sutured manifold with sutured Floer homology given by $V\otimes\widehat{\HFL}(L'')$. Here $V$ is a rank $2$ vector space supported in $A_{L_i''}$ grading $\pm \frac{1}{2}$ and in $A_K$ grading $0$ for all other $K$. Thus, since we are working with vector spaces over $\Z/2$, $\widehat{\HFL}(L'')\cong\widehat{\HFL}(L)$, as required.
\end{proof}

\section{A detection result for Link Floer homology}~\label{2braidsection}
In this section we give another infinite family of detection results for link Floer homology. Let $L_n$ denote the closure of the 2-braid $\sigma^{n}$ together with its braid axis, as shown in Figure~\ref{fig:Ln}, with specified orientations. Let $L_n'$ denote any copy of $L_n$ with arbitrary orientation.

\begin{figure}[ht]
  \includegraphics[width=20cm]{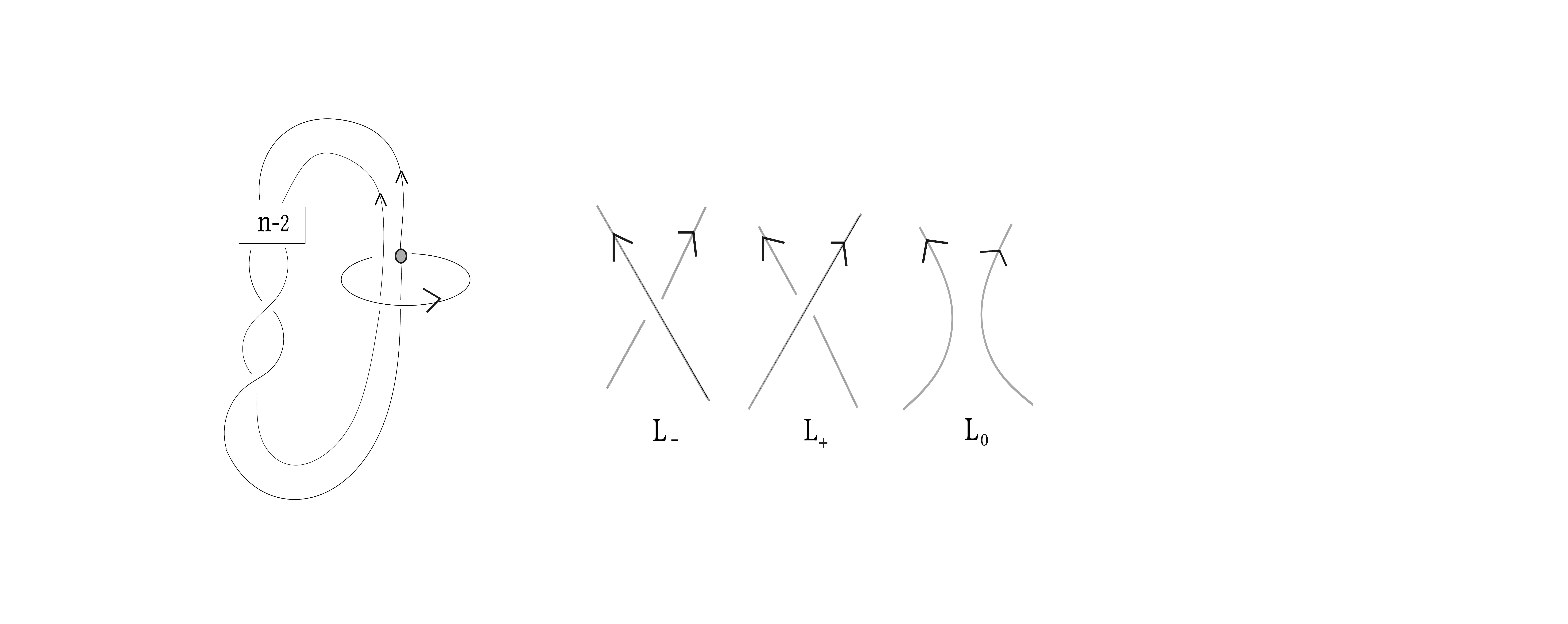}
  \caption{The link $L_n$. The highlighted crossing is one at which we consider a resolution. On its right, a skein triple crossing is drawn.}
  \label{fig:Ln}
\end{figure}

\begin{theorem}\label{2braid}
Link Floer homology detects $L_n'$ for all $n\in\Z$.
\end{theorem}

For $n:=2m+1$ odd we take the first component of $L_n$, $K_1$, to be the braid axis and $K_2$, the second component of $L_n$, to be the braid closure. Note that  $L_{-1}$ is $T(2,4)$ oriented as the boundary of an annulus. We highlight some key features of $\widehat{\HFL}(L_{2m+1})$:

\begin{lemma}\label{lem:oddn}
   Let $m\in\Z$. $\widehat{\HFL}(L_{2m+1})$ satisfies the following conditions:
\begin{enumerate}
    \item The maximum non-trivial $A_1$ grading of $\widehat{HFL}(L_{2m+1}$) is one and $$\sum_{A_2\in\Q}dim(\widehat{HFL}(L_{2m+1}; (1,A_2)) = 2.$$
    \item If $m\geq 0$ the maximum $A_1+A_2$ grading in which $\widehat{\HFL}(L_{2m+1})$ is non-trivial is $m+2$ and \begin{align*}\underset{A_1,A_2\in\Q}{\bigoplus}\widehat{\HFL}(L_{2m+1};A_1+A_2=m+2))\cong\F_0.\end{align*}
    \item If $m< -1$ the maximum value of $A_1+A_2$ grading in which $\widehat{\HFL}(L_{2m+1})$ is non-trivial is $-m$ and \begin{align*}\underset{A_1,A_2\in\Q}{\bigoplus}\widehat{\HFL}(L_{2m+1};A_1+A_2=-m))\cong\F_{-2m-2}.\end{align*}
\end{enumerate}
\end{lemma}

In the proof of part (2) and part (3) of this statement, we use the Skein exact triangle in knot Floer homology. Note that there exist different Maslov grading conventions for knot Floer homology. We use the convention that the Maslov grading for $\widehat{\HFK}(L)$ is $\Z+\frac{n}{2}$-valued, where $n$ is the number of components of $L$. This is the convention used in~\cite{Holomorphicdisksandknotinvariants}.
\begin{proof}

The first part of the first statement follows from Equation~\ref{eq:maxA1grading}. The second part follows from the fact that the second component of $L$ is braided about the first, which is unknotted, combined with Martin's braid axis detection result~\cite[Proposition 1]{martin2022khovanov}.

The second statement can be deduced as follows. Consider the skein exact sequence ~\cite[Equation 8]{Holomorphicdisksandknotinvariants} corresponding to resolving the highlighted crossing in Figure~\ref{fig:Ln}. Here $L_{2m+1}$ is the positive resolution, the zero resolution is the torus knot $T(2,2m+3)$ and the negative resolution is $T(2,2m+1)$ connect sum a positive Hopf link, denoted by $H_+$. Recall that $H_+$ has knot Floer homology given by:\begin{align*}
\widehat{\HFK}(H_+,A)\cong\begin{cases}
    \F_{1/2}&\text{ if }A=1\\
    \F_{-1/2}^2&\text{ if }A=0\\
    \F_{-3/2}&\text{ if }A=-1\\
    0&\text{ otherwise.}
\end{cases}
\end{align*}

Assume now that $m\geq 0$. Note that both $\widehat{\HFK}(T(2,2m+1)\#H_+)$ and $\widehat{\HFK}(T(2,2m+3))$ have maximal non-trivial Alexander grading $m+1$. The generator of $\widehat{\HFK}(T(2,2m+3))$ in Alexander grading $m+1$ has Maslov grading $0$. Following~\cite[Equation 8]{Holomorphicdisksandknotinvariants}, let $V$ denote the bigraded vector space

\begin{equation}\label{eq:strangeV}
    V:=\F_1[1]\oplus\F_0^2[0]\oplus\F_{-1}[-1] 
\end{equation}

where the subscript denotes the Maslov grading and $[A]$ denotes the Alexander grading of the generator. It follows that $\widehat{\HFK}(T(2,2m+3))\otimes V$ is rank one in the maximal non-trivial Alexander grading, $m+2$, with a generator of Maslov grading $1$. Since the exact triangle from~\cite[Equation 8]{Holomorphicdisksandknotinvariants} respects Alexander gradings, it follows that $\widehat{\HFK}(L_{2m+1})$ is rank one in maximal non-trivial Alexander grading $m+2$. Indeed, it follows from the Maslov grading shift formula that $\widehat{\HFK}(L_{2m+1})$ is supported in Maslov grading $1/2$, which directly implies the stated statement.

Assume now that $m< -1$. $\widehat{\HFK}((T(2,2m+1)\#H_+)$ has maximal non-trivial Alexander grading $-m$. The generator in Alexander grading $-m$ has Maslov grading $-2m-\frac{3}{2}$. $\widehat{\HFK}(T(2,2m+3))$ has maximal non-trivial Alexander grading $-m-2$. It follows that $\widehat{\HFK}(T(2,2m+3))\otimes V$ is rank one in the maximal non-trivial Alexander grading, $-m-1$. Since the exact triangle from~\cite[Equation 8]{Holomorphicdisksandknotinvariants} respects Alexander gradings, it follows that $\widehat{\HFK}(L_{2m+1})$ is rank one in maximal non-trivial Alexander grading $-m$. Indeed, it follows from the Maslov grading shift formula -- see~\cite[Theorem 9.1.2]{ozsvath2015grid}, for example -- that $\widehat{\HFK}(L_{2m+1})$ is supported in Maslov grading $-2m-\frac{3}{2}$, which directly implies the stated statement.

\end{proof}

In the odd case, the key step in the proof of Theorem~\ref{2braid} is the following:

\begin{proposition}\label{prop:odd}
Let $n$ be odd. Suppose $L$ is a link with $\widehat{\HFL}(L)\cong\widehat{\HFL}(L_n')$. Then $L$ is isotopic to $L_m'$ for some odd $m$.
\end{proposition}

\begin{proof}
Suppose $L$ is as in the statement of the theorem. Then $L$ is a two component link, with $\lk(L)=\lk(L_n)=2$ since the Conway polynomial, and hence Link Floer homology, detects the linking number of two component links~\cite{hoste1985firstcoefficientoftheconwaypolynomial}. Note that $\widehat{\HFL}(L)$ is of rank $2$ in one of the maximal $A_i$ gradings. Reorder the components of $L$ so that this is the $A_1$ grading. Note that the $A_1$ grading has span $[-1,1]$. It follows from Equation~\ref{eq:maxA1grading} that $L_1$ bounds a disk which intersects $L_2$ -- the other component of $L$ -- twice geometrically or a genus one surface which doesn't intersect $L_2$. However, $\lk(L)\neq 0$, so $L_2$ must intersect any surface bounding $L_1$. It follows that $L_1$ bounds a disk. Thus $L_1$ is an unknot, which is fibered. Note also that the $L$ cannot have reducible exterior, again as $\lk(L)\neq 0$. It follows that $L_2$ is braided about $L_1$ by~\cite[Proposition 1]{martin2022khovanov}. Since $\lk(L)=2$, $L_2$ is a two stranded braid about $L_1$. Thus $L$ is isotopic to $L_m$ for some odd $m$.
\end{proof}

It thus suffices to show the following:

\begin{proposition}
Let $n$ and $m$ be odd. Suppose $\widehat{\HFL}(L_n')\cong\widehat{\HFL}(L_m')$. Then $m=n$.
\end{proposition}

\begin{proof}
Fix $m$ as in the statement of the Proposition. By reversing the orientation of the second component of of $L_m'$ if necessary we can arrange that the two components of $L_m'$ have positive linking number -- that is we can change $L_m'$ to $L_m$. Suppose $n$ is an odd number with $\widehat{\HFL}(L_n')\cong\widehat{\HFL}(L_m')$. Observe that after performing the same changes on $L_n'$ as we did on $L_m'$ we obtain $L_{\pm n}$ and have that $\widehat{\HFL}(L_{\pm n})\cong\widehat{\HFL}(L_m)$. Here we again use the fact that link Floer homology detects the linking number of two component links and~\cite[Equation 2]{BG}. The result then follows from Lemma~\ref{lem:oddn}, together with the fact that Link Floer homology detects $T(2,4)$~\cite[Theorem 3.2]{binns2020knot}.
\end{proof}
Theorem~\ref{2braid} now follows immediately from the preceding two results in the case that $n$ is odd.

For $n:=2m$ even we take the first component of $L_n$, $K_1$ to be the braid axis. As we did earlier, we highlight some key features of $\widehat{\HFL}(L_{2m})$:

\begin{lemma}\label{lem:evenn}
    
Let $m\in\Z$. $\widehat{\HFL}(L_{2m})$ satisfies the following conditions:

\begin{enumerate}
\item The maximum non-trivial $A_1$ grading of $\widehat{HFL}(L_{2m+1}$) is one and $$\sum_{A_2\in\Q}dim(\widehat{HFL}(L_{2m+1}; (1,A_2)) = 4.$$
    \item  If $m\geq 0$ then the maximum value of $A_1+A_2+A_3$ in which $\widehat{\HFL}(L_{2m})$ is non-trivial is $m+2$ and \begin{align*}\underset{A_i\in\Q}{\bigoplus}\widehat{\HFL}(L_{2m};A_1+A_2+A_3=m+2)\cong\F_0.\end{align*}
    \item If $m< 0$ then the the maximum value of $A_1+A_2+A_3$ in which $\widehat{\HFL}(L_{2m})$ is non-trivial is $1-m$ and \begin{align*}\underset{A_i\in\Q}{\bigoplus}\widehat{\HFL}(L_{2m};A_1+A_2+A_3=1-m)\cong\F_{-2m}.\end{align*}
\end{enumerate}

\end{lemma}
\begin{proof}

The first part of the first statement follows from Equation~\ref{eq:maxA1grading}. The second part follows from the fact that the second component of $L$ is braided about the first, which is unknotted, combined with Martin's braid axis detection result~\cite[Proposition 1]{martin2022khovanov}.

The second and the third statements can be deduced as follows. Consider the skein exact sequence~\cite[Equation 8]{Holomorphicdisksandknotinvariants} corresponding to resolving the highlighted crossing in Figure~\ref{fig:Ln}. Here $L_{2m}$ is the positive resolution, the zero resolution is the torus knot $T(2,2m+2)$ and the negative resolution is $T(2,2m)$ connect sum a positive Hopf link.

Assume now that $m\geq 0$. $\widehat{\HFK}(T(2,2m)\#H_+)$ has maximal non-trivial Alexander grading $m+1$. $\widehat{\HFK}(T(2,2m+2))$ has maximal non-trivial Alexander grading $m+1$. The generator in Maslov grading $m+1$ has Maslov grading $0$. Following~\cite[Equation 8]{Holomorphicdisksandknotinvariants}, let $V$ be as in Equation~\ref{eq:strangeV}. We have that $\widehat{\HFK}(T(2,2m+2))\otimes V$ is rank one in the maximal non-trivial Alexander grading, $m+2$, with a generator of Maslov grading $1$. Since the exact triangle from~\cite[Equation 8]{Holomorphicdisksandknotinvariants} respects Alexander gradings, it follows that $\widehat{\HFK}(L_{2m})$ is rank one in maximal non-trivial Alexander grading $m+2$. Indeed, it follows from the Maslov grading shift formula that $\widehat{\HFK}(L_{2m})$ is supported in Maslov grading $1/2$, which directly implies the statement.

Assume now that $m< 0$. $\widehat{\HFK}(T(2,2m)\#H_+)$ has maximal non-trivial Alexander grading $1-m$. The generator in Alexander grading $1-m$ has Maslov grading $-2m+\frac{1}{2}$. $\widehat{\HFK}(T(2,2m+2))$ has maximal non-trivial Alexander grading $-m-1$. Let $V$ be as in Equation~\ref{eq:strangeV}. It follows that $\widehat{\HFK}(T(2,2m+2))\otimes V$ is rank one in the maximal non-trivial Alexander grading, $-m$. Since the exact triangle from~\cite[Equation 8]{Holomorphicdisksandknotinvariants} respects Alexander gradings, it follows that $\widehat{\HFK}(L_{2m})$ is rank one in maximal non-trivial Alexander grading $1-m$. Indeed, it follows from the Maslov grading shift formula that $\widehat{\HFK}(L_{2m};1-m)$ is supported in Maslov grading $-2m+\frac{1}{2}$, which directly implies the statement.

\end{proof}
In the even case, the key step in the proof of Theorem~\ref{2braid} is the following:

\begin{proposition}
Suppose $n$ is even. Suppose $L$ is a link with $\widehat{\HFL}(L)\cong\widehat{\HFL}(L_n')$. Then $L$ is isotopic to $L_m'$ for some even $m$.
\end{proposition}

\begin{proof}
Suppose $L$ is as in the statement of the proposition. Recall that the Conway polynomial determines the co-factors of the symmetric matrix with entries given by $l_{i,j}=\lk(L_i,L_j)=q$ for $i\neq j$, and $l_{i,i}=-\sum_{1\leq j\leq n, j\neq i}\lk(L_i,L_j)$. Since Link Floer homology determines the Conway polynomial, Link Floer homology also determines the cofactors of the given matrix. Observe that for $L_n$ the co-factors of this matrix are $1+2n\neq 0$. Let $K_1$, $K_2$ and $K_3$ be the components of $L$. Suppose $\lk(K_1,K_2)=\lk(K_1,K_3)=0$. Then the co-factors of this matrix are zero. If follows that at least one of $\lk(K_1,K_2),\lk(K_1,K_3)$ is non-zero.

Consider $\widehat{\HFL}(L)$. The span of the $A_1$ grading is $[-1,1]$. It follows from Equation~\ref{eq:maxA1grading} that either: $K_1$ bounds a disk which intersects $L-K_1$ twice geometrically or $K_1$ bounds a genus one surface which does not intersect $L-K_1$. Since at least one of $\lk(K_1,K_2),\lk(K_1,K_3)$ is non-zero, it follows that $K_1$ is an unknot bounding a disk which intersects $L-K_1$ twice.

Now, $\widehat{\HFL}(L)$ is of rank $4$ in the maximal $A_1$ grading and the span of the $A_1$ grading is $[-1,1]$. Suppose that $L$ has reducible exterior. Then we can write $L=L'\sqcup L''$, where $K_1$ is a component of $L'$ and $L''$ has a single component. The K\"unneth formula for the split sum of two links tells us that $2\rank(\widehat{\HFL}(L''))\cdot\rank(\widehat{\HFL}(L';A_1=1))=4$. It follows that $\rank(\widehat{\HFL}(L''))\leq 2$, so that $L''$ is an unknot, $U$. However, the grading information of $\widehat{\HFL}(L_n)$ implies that it is not of the form $\widehat{\HFL}(L'\sqcup U)$ for any $L'$. This is a contradiction, so $L$ has an irreducible exterior.

Since $L$ has irreducible exterior and $K_1$ is fibered, applying~\cite[Proposition 1]{martin2022khovanov}, we find that $K_1$ is a braid axis for $L-K_1$. Indeed, we have that $L-K_1$ is a two stranded braid about $K_1$, since there is a disk bounded by $K_1$ which intersects $L-K_1$ twice. It follows that $L$ is of the form $L_m$ for some even $m$.
\end{proof}

It only remains to show the following:

\begin{proposition}
Let $n$ and $m$ be even. Suppose $\widehat{\HFL}(L_n')\cong\widehat{\HFL}(L_m')$. Then $m=n$.
\end{proposition}

\begin{proof}

Fix $m$ as in the statement of the Proposition. Suppose the braided components of $L_m'$ are oriented in parallel. By reversing the orientation of the braid axis component we can arrange that the braided components of $L_m'$ have positive linking number with the first component -- that is we can change $L_m'$ to $L_m$. Suppose $n$ is an even number with $\widehat{\HFL}(L_n')\cong\widehat{\HFL}(L_m')$.  Observe that after performing the same changes on $L_n'$ as we did on $L_m'$ we obtain a link $L_{ n}''$ and have that $\widehat{\HFL}(L_{ n}'')\cong\widehat{\HFL}(L_m)$. The result then follows from Lemma ~\ref{lem:evenn}.

Suppose the braided components of $L_m'$ are not oriented in parallel. By reversing the orientation of a braided component, we can arrange that the braided components of $L_m'$ have positive linking number with the braid axis component  -- that is we can change $L_m'$ to $L_m$. Suppose $n$ is an even number with $\widehat{\HFL}(L_n')\cong\widehat{\HFL}(L_m')$.  Observe that after performing the same changes on $L_n'$ as we did on $L_m'$ we obtain a link $L_{n}''$ and have that $\widehat{\HFL}(L_{n}'')\cong\widehat{\HFL}(L_m)$. The result then follows from Lemma~\ref{lem:evenn}.
\end{proof}

Theorem~\ref{2braid} now follows immediately from the preceding lemmas and propositions.

\bibliographystyle{plain}
\bibliography{citations}
\end{document}